\documentclass[12pt]{amsart}
\usepackage{a4wide}
\usepackage{graphicx,eepic, bm, times}
\usepackage{amsthm,amsmath,amsfonts,amssymb}
\usepackage{multirow}

\newtheorem{lemma}{Lemma}[section]

\newtheorem{prop}[lemma]{Proposition}
\newtheorem{thm}[lemma]{Theorem}
\newtheorem{cor}[lemma]{Corollary}
\theoremstyle{definition}

\theoremstyle{remark}
\newtheorem{remark}[lemma]{Remark}

\numberwithin{equation}{section} \numberwithin{table}{section}

\begin{document}

\title[On small bases which admit countably many expansions]{On small bases which admit countably many expansions}
\dedicatory{Dedicated to P. Erd\H{o}s on the 100th anniversary of his birth}
\author{Simon Baker}
\address{
School of Mathematics, The University of Manchester,
Oxford Road, Manchester M13 9PL, United Kingdom. E-mail:
simon.baker@manchester.ac.uk}

\date{\today}
\subjclass[2010]{11A63, 37A45}
\keywords{Beta-expansion, non-integer base} 

\begin{abstract}
Let $q\in(1,2)$ and $x\in[0,\frac1{q-1}]$. We say that a sequence $(\epsilon_i)_{i=1}^{\infty}\in\{0,1\}^{\mathbb{N}}$ is an expansion of $x$ in base~$q$ (or a $q$-expansion) if
\[
x=\sum_{i=1}^{\infty}\epsilon_iq^{-i}.
\]
Let $\mathcal{B}_{\aleph_{0}}$ denote the set of $q$ for which there exists $x$ with exactly $\aleph_{0}$ expansions in base $q$. In \cite{EHJ} it was shown that $\min\mathcal{B}_{\aleph_{0}}=\frac{1+\sqrt{5}}{2}.$ In this paper we show that the smallest element of $\mathcal{B}_{\aleph_{0}}$ strictly greater than $\frac{1+\sqrt{5}}{2}$ is $q_{\aleph_{0}}\approx1.64541$, the appropriate root of $x^6=x^4+x^3+2x^2+x+1$. This leads to a full dichotomy for the number of possible $q$-expansions for $q\in (\frac{1+\sqrt{5}}{2},q_{\aleph_{0}})$. We also prove some general results regarding $\mathcal{B}_{\aleph_{0}}\cap[\frac{1+\sqrt{5}}{2},q_{f}],$ where $q_{f}\approx 1.75488$ is the appropriate root of $x^{3}=2x^{2}-x+1.$ Moreover, the techniques developed in this paper imply that if $x\in [0,\frac{1}{q-1}]$ has uncountably many $q$-expansions then the set of $q$-expansions for $x$ has cardinality equal to that of the continuum, this proves that the continuum hypothesis holds when restricted to this specific case.
\end{abstract}

\maketitle

\section{Introduction}

Let ${q}\in(1,2)$ and $I_{q}=[0,\frac{1}{{q}-1}]$. Each $x\in I_{q}$ has an expansion of the form
\begin{equation}
\label{beta equation}
x=\sum_{i=1}^{\infty}\frac{\epsilon_{i}}{{q}^{i}},
\end{equation} for some $(\epsilon_{i})_{i=1}^{\infty}\in\{0,1\}^{\mathbb{N}}$. We call such a sequence a \textit{$q$-expansion} of $x,$ when $(\ref{beta equation})$ holds we will adopt the notation $x=(\epsilon_{1},\epsilon_{2},\ldots)_{q}$. Expansions in non-integer bases were pioneered in the papers of R\'enyi \cite{Renyi} and Parry \cite{Parry}.

Given $x\in I_{q}$ we denote the set of ${q}$-expansions of $x$ by $\Sigma_{q}(x)$, i.e., $$\Sigma_{q}(x)=\Big\{(\epsilon_{i})_{i=1}^{\infty}\in \{0,1\}^{\mathbb{N}} : \sum_{i=1}^{\infty}\frac{\epsilon_{i}}{{q}^{i}}=x\Big\}.$$ The endpoints of $I_{q}$ always have a unique $q$-expansion, typically an element of $(0,\frac{1}{q-1})$ will have a nonunique $q$-expansion. In \cite{Erdos} it was shown that for $q\in(1,\frac{1+\sqrt{5}}{2})$ the set $\Sigma_{q}(x)$ is uncountable for all $x\in (0,\frac{1}{{q}-1})$. When $q=\frac{1+\sqrt{5}}{2}$ it was shown in \cite{SidVer} that every $x\in(0,\frac{1}{q-1})$ has uncountably many $q$-expansions unless $x=\frac{(1+\sqrt{5})n}{2}\bmod1$, for some $n\in\mathbb{Z}$, in which case $\Sigma_{q}(x)$ is infinite countable. In \cite{Sid} it was shown that for $q\in(\frac{1+\sqrt{5}}{2},2)$ the set $\Sigma_{q}(x)$ is uncountable for almost every $x\in (0,\frac{1}{q-1})$.  Furthermore, if $q\in(\frac{1+\sqrt{5}}{2},2)$ then it was shown in \cite{DaKa} that there always exists $x\in (0,\frac{1}{{q}-1})$ with a unique $q$-expansion. 

In this paper we will be interested in the set of $q\in(1,2)$ for which there exists $x\in (0,\frac{1}{q-1})$ satisfying card $\Sigma_{q}(x)=\aleph_{0}$. More specifically, we will be interested in the set $$\mathcal{B}_{\aleph_{0}}:=\Big\{{q}\in(1,2)\Big| \textrm{ there exists } x\in \Big(0,\frac{1}{{q}-1}\Big) \textrm{ satisfying} \text{ card }\Sigma_{q}(x)=\aleph_{0}\Big\}.$$ In \cite{EHJ} it was shown $\min\mathcal{B}_{\aleph_{0}}=\frac{1+\sqrt{5}}{2}.$ We can define $\mathcal{B}_{k}$ in an analogous way for all $k\geq 1.$ It was first shown in \cite{EJ} that $\mathcal{B}_{k}\neq \emptyset$ for all $k\geq 2,$ this was later improved upon in \cite{Sid1} where it was shown that for each $k\in\mathbb{N}$ there exists $\gamma_{k}>0$ such that $(2-\gamma_{k},2)\subset \mathcal{B}_{j}$ for all $1\leq j\leq k$. Combining the results stated in \cite{Sid1} and \cite{BakerSid} the following theorem is shown to hold.

\begin{thm}
\label{Finite case thm}
\begin{itemize}
	\item The smallest element of $\mathcal{B}_{2}$ is
$$
{q}_{2}\approx 1.71064,
$$
the appropriate root of $x^{4}=2x^{2}+x+1$.
	\item For $k\geq 3$ the smallest element of $\mathcal{B}_{k}$ is
$$
{q}_{f}\approx 1.75488,
$$
the appropriate root of $x^{3}=2x^{2}-x+1$. 
\item Moreover, the first element of $\mathcal{B}_{2}$ strictly greater than $q_{2}$ is $q_{f}$. 
\end{itemize}
\end{thm}
In this paper we will show that the following theorem holds.

\begin{thm}
\label{Main thm}
The smallest element of $\mathcal{B}_{\aleph_{0}}$ strictly greater than $\frac{1+\sqrt{5}}{2}$ is $$q_{\aleph_{0}}\approx 1.64541,$$ the appropriate root of $x^6=x^4+x^3+2x^2+x+1.$
\end{thm}
This answers a question originally posed in \cite{Sid1}. The following corollary is an immediate consequence of Theorem \ref{Finite case thm}, Theorem \ref{Main thm} and our earlier remarks, it implies a full dichotomy for the number of possible $q$-expansions for $q\in (\frac{1+\sqrt{5}}{2},q_{\aleph_{0}}).$

\begin{cor}
Let $q\in (\frac{1+\sqrt{5}}{2},q_{\aleph_{0}}),$  then there exists $x\in (0,\frac{1}{q-1})$ such that $\Sigma_{q}(x)$ is uncountable and there exists $x\in(0,\frac{1}{q-1})$ with a unique $q$-expansion, moreover, for any $x\in(0,\frac{1}{q-1})$ the set $\Sigma_{q}(x)$ is either uncountable or a singleton set.
\end{cor}
Before stating the theory behind Theorem \ref{Main thm} we shall outline our method of proof, this will help to motivate the following sections. If $q\in\mathcal{B}_{\aleph_{0}},$ then as we will see, there must exist $x\in I_{q}$ for which $\Sigma_{q}(x)$ takes a highly nontrivial structure, in the following sections we refer to these $x$ as $q$ null infinite points. If $I_{q}$ contains a $q$ null infinite point and $q\in [\frac{1+\sqrt{5}}{2},q_{f})\setminus\{q_{2}\},$ then $q$ must satisfy certain strong algebraic properties. Once these properties are appropriately formalised it is apparent that they cannot be satisfied for $q'$ sufficiently close to $q,$ this implies that there exists $\delta>0$ such that $\mathcal{B}_{\aleph_{0}}\cap (\frac{1+\sqrt{5}}{2},\frac{1+\sqrt{5}}{2}+\delta)=\emptyset,$ and more generally that $\mathcal{B}_{\aleph_{0}}\cap([\frac{1+\sqrt{5}}{2},q_{f})\setminus\{q_{2}\})$ is a discrete set. The $\delta$ produced by our method in fact turns out to be optimal. We remark that $q_{\aleph_{0}}\in \mathcal{B}_{\aleph_{0}}$ was already known to Hare and Sidorov \cite{HareSid}, moreover, numerical experiments done by Hare seemed to suggest $q_{\aleph_{0}}$ was the smallest element of $\mathcal{B}_{\aleph_{0}}$ strictly greater than $\frac{1+\sqrt{5}}{2}.$

In Section \ref{Section 2} we establish several technical results that will be used in Section \ref{proof section} where we prove Theorem \ref{Main thm}. In Section \ref{Final discussion} we prove several results that arose naturally from our proof of Theorem \ref{Main thm}. In particular, we prove that for all $q\in (1,2),$ if $x\in I_{q}$ satisfies $\Sigma_{q}(x)$ is uncountable, then it must have cardinality equal to that of the continuum, as such the continuum hypothesis holds for this specific case, this answers a question attributed to Erd\H{o}s. We also show that $\mathcal{B}_{\aleph_{0}}\cap([\frac{1+\sqrt{5}}{2},q_{f})\setminus\{q_{2}\})$ is a discrete set, and propose a method by which we can determine whether a typical $q\in[\frac{1+\sqrt{5}}{2},q_{f})\setminus\{q_{2}\}$ is an element of $\mathcal{B}_{\aleph_{0}}.$ Finally in Section \ref{Open problems} we pose some open questions.

\section{Preliminaries}
\label{Section 2}
We begin by recalling some standard results. In what follows we fix $T_{{q},0}(x)={q} x$ and $T_{{q},1}(x)={q} x -1$, we typically denote an element of $\bigcup_{n=0}^{\infty}\{T_{{q},0},T_{{q},1}\}^{n}$ by $a;$ here $\{T_{{q},0},T_{{q},1}\}^{0}$ denotes the set consisting of the identity map. Moreover, if $a= (a_{1},\ldots ,a_{n})$ we shall use $a(x)$ to denote $(a_{n}\circ \cdots  \circ a_{1})(x)$ and $|a|$ to denote the length of $a$. Given $a\in \bigcup_{n=0}^{\infty}\{T_{{q},0},T_{{q},1}\}^{n}$ and $q'\neq q,$ we can identify $a$ with an element of $\bigcup_{n=0}^{\infty}\{T_{{q'},0},T_{{q'},1}\}^{n}$ by replacing each $T_{q,i}$ term in $a$ with a $T_{q',i}$ term. By an abuse of notation we also denote the element of $\bigcup_{n=0}^{\infty}\{T_{{q'},0},T_{{q'},1}\}^{n}$ attained through this identification by $a$, whether we are interpreting $a$ as an element of $\bigcup_{n=0}^{\infty}\{T_{{q},0},T_{{q},1}\}^{n}$ or $\bigcup_{n=0}^{\infty}\{T_{{q'},0},T_{{q'},1}\}^{n}$ will be clear from the context. We will make regular use of this identification in Section \ref{Final discussion}.

We let $$\Omega_{q}(x)=\Big\{(a_{i})_{i=1}^{\infty}\in \{T_{{q},0},T_{{q},1}\}^{\mathbb{N}}:(a_{n}\circ \cdots \circ a_{1})(x)\in I_{q}
 \textrm{ for all } n\in\mathbb{N}\Big\}.$$
The significance of $\Omega_{q}(x)$ is made clear by the following lemma.

\begin{lemma}
\label{Bijection lemma}
$\text{card }\Sigma_{q}(x)=\text{card }\Omega_{q}(x)$ where our bijection identifies $(\epsilon_{i})_{i=1}^{\infty}$ with $(T_{{q},\epsilon_{i}})_{i=1}^{\infty}$.
\end{lemma}
The proof of Lemma \ref{Bijection lemma} is contained within \cite{Baker}. It is an immediate consequence of Lemma \ref{Bijection lemma} that we can interpret Theorem \ref{Main thm} in terms of $\Omega_{q}(x)$ rather than $\Sigma_{q}(x)$. Throughout the course of our proof of Theorem \ref{Main thm} we will frequently switch between $\Sigma_{q}(x)$ and the dynamical interpretation of $\Sigma_{q}(x)$ provided by Lemma \ref{Bijection lemma}, often considering $\Omega_{q}(x)$ will help our exposition.

An element $x\in I_{q}$ satisfies $T_{{q},0}(x)\in I_{q}$ and $T_{{q},1}(x)\in I_{q}$ if and only if $x\in[\frac{1}{q},\frac{1}{{q}({q}-1)}]$. Furthermore, if $\text{card }\Sigma_{q}(x)>1$ or equivalently $\text{card }\Omega_{q}(x)>1$, then there exists a unique minimal sequence of transformations $a$ such that $a(x)\in[\frac{1}{q},\frac{1}{{q}({q}-1)}]$. Throughout this paper when we speak of a finite sequence being minimal we mean minimal amongst $\bigcup_{n=0}^{\infty}\{T_{{q},0},T_{{q},1}\}^{n}$ with respect to length. In what follows we let $S_{q}:=[\frac{1}{q},\frac{1}{{q}({q}-1)}];$ $S_q$ is usually referred to as the \textit{switch region}. If $x\in(0,\frac{1}{q-1})$ and $a$ is a finite sequence of transformations satisfying $a(x)\in S_{q},$ then we say that the sequence $a$ is a \textit{branching sequence for $x$} and $a(x)$ is a \textit{branching point of $x$.}

In what follows we denote the set of $x\in I_{q}$ with unique $q$-expansion by $U_{q},$ i.e. $$U_{q}=\Big\{x\in I_{q}| \textrm{ card }\Sigma_{q}(x)=1\Big\}.$$ The following lemma is a consequence of \cite[Theorem~2]{GlenSid}.

\begin{lemma}
\label{Unique expansions lemma}
Let ${q}\in(\frac{1+\sqrt{5}}{2},{q}_{f}]$, then $$U_{q}=\Big\{(0^{k}(10)^{\infty})_{q},(1^{k}(10)^{\infty})_{q},0,\frac{1}{q-1}\Big\},$$ where $k\ge0$.
\end{lemma}
In Lemma \ref{Unique expansions lemma} we have adopted the notation $(\epsilon_{1},\dots,\epsilon_{n})^{k}$ to denote the concatenation of $(\epsilon_{1},\dots ,\epsilon_{n})\in \{0,1\}^{n}$ by itself $k$ times and $(\epsilon_{1},\dots ,\epsilon_{n})^{\infty}$ to denote the element of $\{0,1\}^{\mathbb{N}}$ obtained by concatenating $(\epsilon_{1},\ldots,\epsilon_{n})$ by itself infinitely many times, we will use this notation throughout. Lemma \ref{Unique expansions lemma} will be a useful tool when it comes to showing that $(\frac{1+\sqrt{5}}{2},q_{\aleph_{0}})\cap \mathcal{B}_{\aleph_{0}}=\emptyset.$

\subsection{Branching argument}
To prove Theorem \ref{Main thm} we use a variation of the branching argument that first appeared in \cite{Sid2}. Before giving details of our approach we describe the construction given in \cite{Sid2}.
\subsubsection{Construction of the branching tree corresponding to $x$}

We define the \textit{branching tree corresponding to $x$} as follows. Suppose $x$ satisfies $\text{card }\Omega_{q}(x)=1,$ then we define the branching tree corresponding to $x$ to be an infinite horizontal line. If $x$ satisfies $\text{card } \Omega_{q}(x)>1,$ then there exists a unique minimal branching sequence $a,$ we depict this choice of transformation by a horizontal line of finite length that then bifurcates with an upper and lower branch. The upper branch corresponds to the sequence of transformations obtained by concatenating the branching sequence $a$ by $T_{q,0}$ and the lower branch corresponds to the sequence of transformations obtained by concatenating the branching sequence $a$ by $T_{q,1}$. If $T_{q,i}(a(x))$ satisfies $\Omega_{q}(T_{q,i}(a(x)))=1$ then we extend the branch corresponding to $T_{q,i}(a(x))$ by an infinite horizontal line. If $\Omega_{q}(T_{q,i}(a(x)))>1$ then there exists a unique minimal branching sequence for $T_{q,i}(a(x))$ that we call $a',$ we depict this choice of transformation by extending the branch corresponding to $T_{q,i}(a(x))$ by a horizontal line of finite length that then bifurcates, again the upper branch corresponds to concatenating $a'$ by $T_{q,0},$ and the lower branch corresponds to concatenating $a'$ by $T_{q,1}.$ Repeatedly applying these rules to succesive branching points of $x$ we construct an infinite tree which we refer to as the branching tree corresponding to $x$. We refer the reader to Figure \ref{fig1} for a diagram illustrating the construction of the branching tree corresponding to $x$. Where appropriate we denote the branching tree corresponding to $x$ by $\mathcal{T}(x)$. The branching tree corresponding to $x$ is referred to as the branching compactum in \cite{Sid2}. 
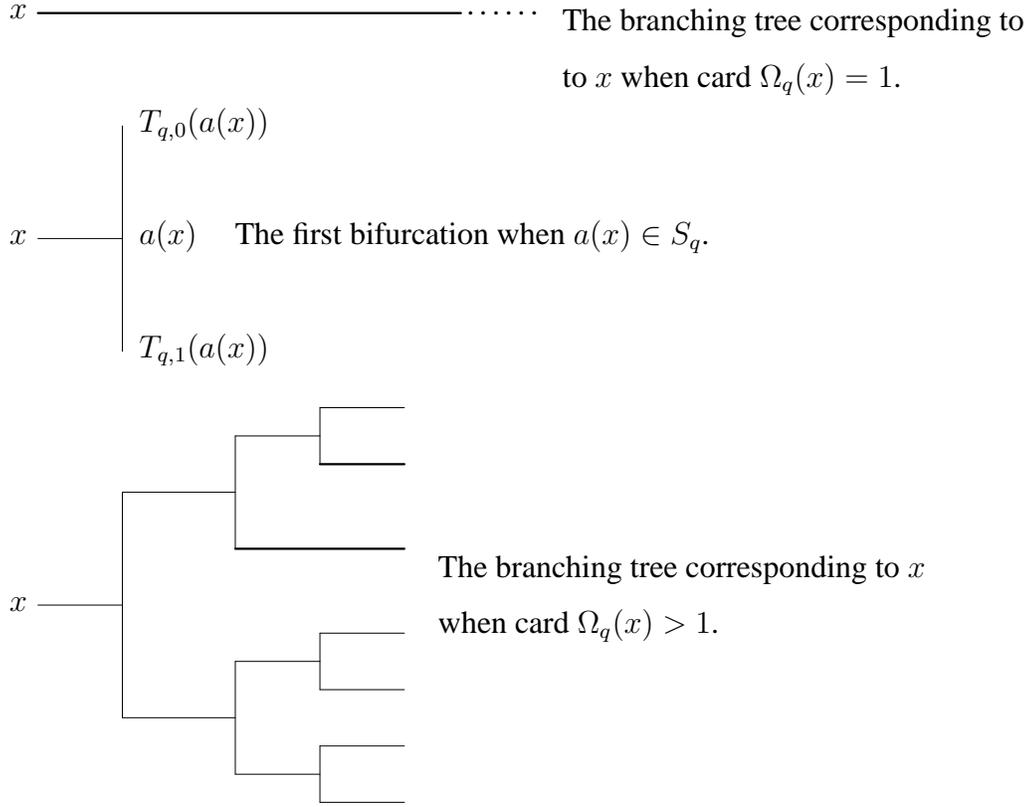
\begin{figure}[t]
\centering \unitlength=0.75mm
\begin{picture}(130,150)(0,0)
\thinlines
\path(-15,105)(0,105)
\path(0,105)(0,125)
\path(0,105)(0,85)
\put(61,144.8){\ldots\ldots}
\put(20,104){The first bifurcation when $a(x)\in S_{q}$.}
\put(78,142){The branching tree corresponding to}
\put(78,132){to $x$ when $\text{card }\Omega_{q}(x)=1$.}
\put(56,45){The branching tree corresponding to $x$}
\put(56,35){when $\text{card }\Omega_{q}(x)>1$.}
\put(-20,144){$x$}
\put(-20,104){$x$}
\put(3,104){$a(x)$}
\put(3,124){$T_{q,0}(a(x))$}
\put(3,84){$T_{q,1}(a(x))$}
\put(-20,39){$x$}
\path(-15,40)(0,40)
\path(0,20)(0,60)
\path(0,60)(20,60)
\path(0,20)(20,20)
\path(20,30)(20,10)
\path(20,70)(20,50)
\path(20,70)(35,70)
\path(20,10)(35,10)
\path(20,30)(35,30)
\path(35,15)(35,5)
\path(35,75)(35,65)
\path(35,35)(35,25)
\path(35,25)(50,25)
\path(35,5)(50,5)
\path(35,15)(50,15)
\path(35,35)(50,35)
\path(35,75)(50,75)
\thicklines
\path(-15,145)(60,145)
\path(20,50)(50,50)
\path(35,65)(50,65)
\end{picture}
\caption{The construction of the branching tree corresponding to $x$}
    \label{fig1}
\end{figure}
\begin{remark}
It is immediate from the construction of $\mathcal{T}(x)$ that there is a bijection between the space of infinite paths in $\mathcal{T}(x)$ and $\Omega_{q}(x),$ which by Lemma \ref{Bijection lemma} implies there is also a bijection between the space of infinite paths in $\mathcal{T}(x)$ and $\Sigma_{q}(x).$
\end{remark}

\subsubsection{Construction of the infinite branching tree corresponding to $x$}
We now give details of our variation of the above construction which will be more suited towards our exposition. Suppose $x\in I_{q}$ satisfies $\Omega_{q}(x)$ is infinite or equivalently $\Sigma_{q}(x)$ is infinite, we define the \textit{infinite branching tree corresponding to $x$} as follows. If for each branching point of $x,$ $a(x),$ we have $\text{card }\Omega_{q}(T_{q,i}(a(x)))<\infty,$ for some $i\in\{0,1\},$ then we define the infinite branching tree corresponding to $x$ to be an infinite horizontal line. If this is not the case then there exists a branching point $a(x)$ such that $\Omega_{q}(T_{q,0}(a(x)))$ and $\Omega_{q}(T_{q,1}(a(x)))$ are both infinite. Taking $a$ to be the unique minimal branching sequence of $x$ for which $\Omega_{q}(T_{q,0}(a(x)))$ and $\Omega_{q}(T_{q,1}(a(x)))$ are both infinite we draw a horizontal line of finite length which bifurcates, the upper branch corresponds to $T_{q,0}(a(x))$ and the lower branch corresponds to $T_{q,1}(a(x)).$ We then extend the branch corresponding to $T_{q,i}(a(x))$ in accordance with the same rules, i.e., if for each branching point of $T_{q,i}(a(x)),$ $a'(T_{q,i}(a(x))),$  we have $\text{card }\Omega_{q}(a'(T_{q,i}(a(x))))<\infty,$ for some $i\in\{0,1\},$ we extend the branch corresponding to $T_{q,i}(a(x))$ by an infinite horizontal line, and if there exists a branching point of $T_{q,i}(a(x))$, $a'(T_{q,i}(a(x))),$ such that both $\Omega_{q}(T_{q,0}(a'(T_{q,i}(a(x)))))$ and $\Omega_{q}(T_{q,1}(a'(T_{q,i}(a(x)))))$ are infinite, we extend the branch corresponding to $T_{q,i}(a(x))$ by a finite horizontal line that then bifurcates, with upper branching corresponding to $T_{q,0}(a'(T_{q,i}(a(x)))),$ and lower branch corresponding to $T_{q,1}(a'(T_{q,i}(a(x)))).$ Repeatedly applying these rules to each upper and lower branch of our construction we obtain an infinite tree that we refer to as the infinite branching tree corresponding to $x$. We refer the reader to Figure \ref{fig2} for a diagram illustrating the construction of the branching tree corresponding to $x$. Where appropriate we denote the infinite branching tree corresponding to $x$ by $\mathcal{T}_{\infty}(x)$. If $\mathcal{T}_{\infty}(x)$ contains at least one bifurcation then every branch except the initial horizontal branch begins at a point where $\mathcal{T}_{\infty}(x)$ bifurcates, we refer to this point as the \textit{root of the branch}. It is clear from the construction of $\mathcal{T}_{\infty}(x)$ that the root of a branch can be identified with a branching point of $x$.

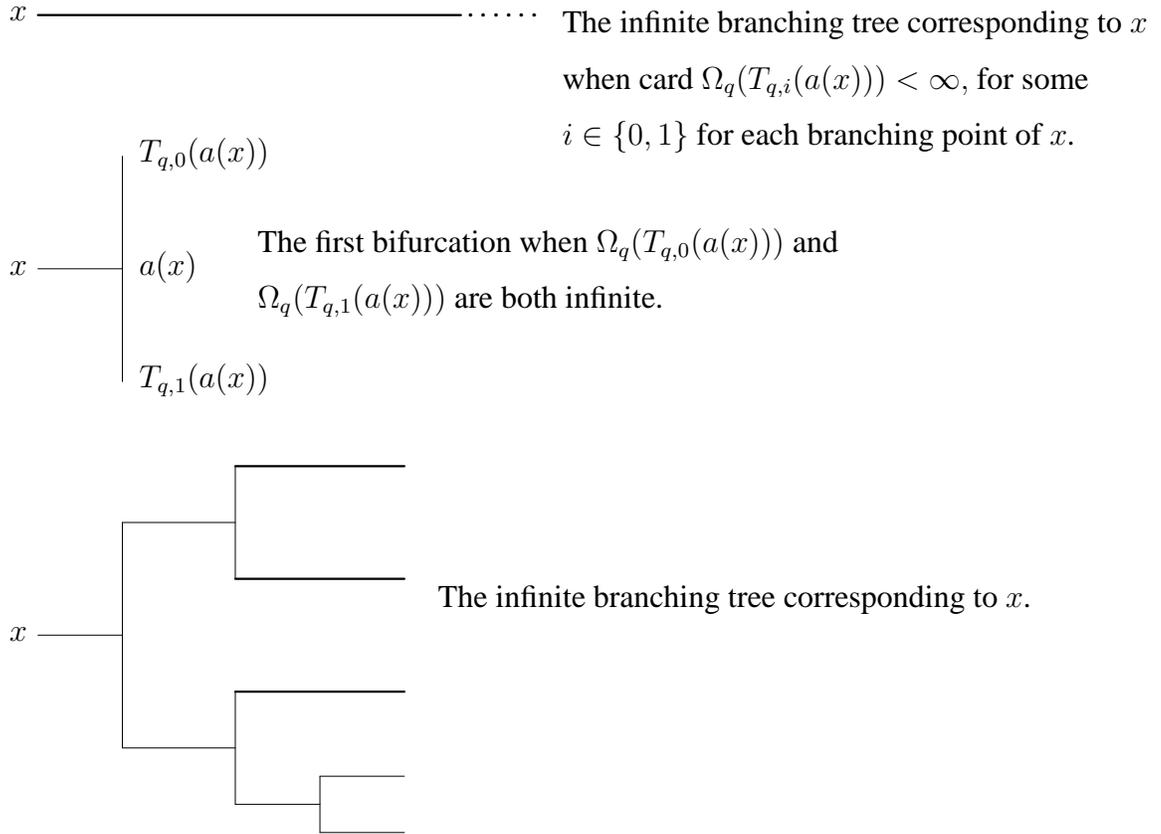
\begin{figure}[t]
\centering \unitlength=0.75mm
\begin{picture}(130,150)(0,0)
\thinlines
\path(-15,100)(0,100)
\path(0,100)(0,120)
\path(0,100)(0,80)
\put(61,144.8){\ldots\ldots}
\put(24,103){The first bifurcation when $\Omega_{q}(T_{q,0}(a(x)))$ and }
\put(24,93){$\Omega_{q}(T_{q,1}(a(x)))$ are both infinite.}
\put(78,142){The infinite branching tree corresponding to $x$ }
\put(78,132){when $\text{card }\Omega_{q}(T_{q,i}(a(x)))<\infty,$ for some }
\put(78,122){$i\in\{0,1\}$ for each branching point of $x$.}
\put(56,40){The infinite branching tree corresponding to $x$.}
\put(-20,144){$x$}
\put(-20,99){$x$}
\put(3,99){$a(x)$}
\put(3,119){$T_{q,0}(a(x))$}
\put(3,79){$T_{q,1}(a(x))$}
\put(-20,34){$x$}
\path(-15,35)(0,35)
\path(0,15)(0,55)
\path(0,55)(20,55)
\path(0,15)(20,15)
\path(20,25)(20,5)
\path(20,65)(20,45)

\path(20,5)(35,5)

\path(35,10)(35,0)

\path(35,0)(50,0)
\path(35,10)(50,10)

\thicklines
\path(-15,145)(60,145)
\path(20,45)(50,45)
\path(20,65)(50,65)
\path(20,25)(50,25)
\end{picture}
\caption{The construction of the infinite branching tree corresponding to $x$}
    \label{fig2}
\end{figure}

\begin{remark}
Every infinite path in $\mathcal{T}_{\infty}(x)$ can be identified with a unique element of $\Omega_{q}(x)$. However, unlike $\mathcal{T}(x)$ not every element of $\Omega_{q}(x)$ necessarily corresponds to a unique infinite path in $\mathcal{T}_{\infty}(x).$
\end{remark}

If $x$ satisfies $\Omega_{q}(x)$ is infinite and for each branching point $a(x)$ we have $\text{card }\Omega_{q}(T_{q,i}(a(x)))<\infty,$ for some $i\in\{0,1\},$ i.e., the case where the infinite branching tree is an infinite horizontal line, then we refer to $x$ as a \textit{$q$ null infinite point}. It is an immediate consequence of our definition that if $x$ is a $q$ null infinite point then $\text{card } \Omega_{q}(x)= \text{card } \Sigma_{q}(x)=\aleph_{0}.$

\begin{remark}
For $q\in (\frac{1+\sqrt{5}}{2},q_{f})\setminus\{q_{2}\},$ it is a consequence of Theorem \ref{Finite case thm} that if $x$ satisfies $\Omega_{q}(x)$ is infinite, then at each branching point $a(x),$ either both $\Omega_{q}(T_{q,0}(a(x)))$ and $\Omega_{q}(T_{q,1}(a(x)))$ are infinite or one of them is infinite and one of them is a singleton set, i.e, $T_{q,i}(a(x))\in U_{q}$ for some $i\in\{0,1\}$. As such, for $q\in (\frac{1+\sqrt{5}}{2},q_{f})\setminus\{q_{2}\}$ we may interpret $\mathcal{T}_{\infty}(x)$ as the infinite tree obtained from $\mathcal{T}(x)$ if we remove all branches that end in infinite horizontal lines.
\end{remark}

\begin{remark}
The case where $x$ is a $q$ null infinite point is of particular importance. By Theorem \ref{Finite case thm} it follows that for $q\in (\frac{1+\sqrt{5}}{2},q_{2})\cup (q_{2},q_{f}),$ if $x$ is a $q$ null infinite point then for each branching point of $x,$ $a(x),$ we must have $\text{card }\Omega_{q}(T_{q,i}(a(x)))=\aleph_{0}$ and $\text{card }\Omega_{q}(T_{q,1-i}(a(x)))=1,$ for some $i\in\{0,1\}$. 
\end{remark}
As the following proposition shows, it is in fact the case that whenever $q\in \mathcal{B}_{\aleph_{0}},$ then $(0,\frac{1}{q-1})$ contains a $q$ null infinite point.

\begin{prop}
\label{Branching prop}
Suppose $q\in \mathcal{B}_{\aleph_{0}},$ then $(0,\frac{1}{q-1})$ contains a $q$ null infinite point. 
\end{prop}

\begin{proof}
If $q\in \mathcal{B}_{\aleph_{0}}$ there exists $x\in (0,\frac{1}{q-1})$ satisfying $\text{card }\Omega_{q}(x)=\aleph_{0}.$  If $x$ is a $q$ null infinite point then we are done, let us assume this is not the case and that $\mathcal{T}_{\infty}(x)$ contains at least one bifurcation. If each branch of $\mathcal{T}_{\infty}(x)$ was to always bifurcate then $\mathcal{T}_{\infty}(x)$ would be the full binary tree, as each infinite path in $\mathcal{T}_{\infty}(x)$ can be identified with a unique element of $\Omega_{q}(x)$ and the set of infinite paths in the full binary tree has cardinality equal to the continuum we would have $\text{card }\Omega_{q}(x)=2^{\aleph_{0}},$ a contradiction. As such there must exist at least one branch that no longer bifurcates, by considering the root of this branch and the corresponding branching point $a(x)\in S_{q},$ either $T_{q,0}(a(x))$ or $T_{q,1}(a(x))$ must be a $q$ null infinite point.
\end{proof}
To prove Theorem \ref{Main thm} we first of all show that $(\frac{1+\sqrt{5}}{2},q_{\aleph_{0}})\cap \mathcal{B}_{\aleph_{0}}=\emptyset,$ we do this by contradiction. By Proposition \ref{Branching prop} if $q\in \mathcal{B}_{\aleph_{0}}$ then  $(0,\frac{1}{q-1})$ contains a $q$ null infinite point, by studying $q$ null infinite points we will be able to derive our desired contradiction.

\section{Proof of Theorem \ref{Main thm}}
\label{proof section}
Our proof of Theorem \ref{Main thm} will be split into two parts, we begin by showing that $(\frac{1+\sqrt{5}}{2},q_{\aleph_{0}})\cap \mathcal{B}_{\aleph_{0}}=\emptyset,$ we then explicitly construct an $x\in I_{q_{\aleph_{0}}}$ for which $\text{card }\Omega_{q_{\aleph_{0}}}(x)=\aleph_{0}$.

\subsection{Proof that $(\frac{1+\sqrt{5}}{2},q_{\aleph_{0}})\cap \mathcal{B}_{\aleph_{0}}=\emptyset$}
To show that $(\frac{1+\sqrt{5}}{2},q_{\aleph_{0}})\cap \mathcal{B}_{\aleph_{0}}=\emptyset$ it is useful to consider the following interval: $$J_{q}:=\Big[\frac{q+q^{2}}{q^{4}-1},\frac{1+q^{3}}{q^{4}-1}\Big].$$ The following identities hold 
\begin{equation}
\label{period 4 identities}
T_{q,1}\Big(T_{q,0}\Big(\frac{q+q^{2}}{q^{4}-1}\Big)\Big)= \frac{1+q^{3}}{q^{4}-1}\textrm{ and } T_{q,0}\Big(T_{q,1}\Big(\frac{1+q^{3}}{q^{4}-1}\Big)\Big)= \frac{q+q^{2}}{q^{4}-1},
\end{equation} it is an immediate consequence of (\ref{period 4 identities}) that $\frac{q+q^{2}}{q^{4}-1}=((0110)^{\infty})_{q}$ and $\frac{1+q^{3}}{q^{4}-1}=((1001)^{\infty})_{q}.$ The endpoints of $J_{q}$ are contained within a $4$-cycle $$\Big\{((0110)^{\infty})_{q},((1100)^{\infty})_{q},((1001)^{\infty})_{q},((0011)^{\infty})_{q}\Big\}.$$ For $q>q_{f}$ this cycle it is a subset of $U_{q},$ moreover, it is the first $4$-cycle to be a subset of $U_{q}$, see \cite{AllClaSid}. The significance of the interval $J_{q}$ is made clear by the following lemma.

\begin{lemma}
\label{inner switch lemma}
Let $q\in (\frac{1+\sqrt{5}}{2},q_{f}]$. Suppose $x\in I_{q}$ satisfies $\text{card }\Omega_{q}(x)>1,$ then there exists a finite sequence of transformations $a$ such that $a(x)\in J_{q}.$
\end{lemma}

\begin{proof}
It is a simple exercise to show that if $q\in (\frac{1+\sqrt{5}}{2},q_{f}]$ then $J_{q}\subseteq S_{q}$ with equality if and only if $q=q_{f}.$ Let $x\in I_{q}$ satisfy $\text{card } \Omega_{q}(x)>1,$ then there exists a finite sequence of transformations $a$ such that $a(x)\in S_{q}.$ If $q=q_{f}$ then we may immediately conclude our result, as such in what follows we assume $q\in (\frac{1+\sqrt{5}}{2},q_{f}).$ If $q\in(\frac{1+\sqrt{5}}{2},2),$ then $S_{q}\subset (\frac{1}{q^{2}-1},\frac{q}{q^{2}-1}).$ The significance of the points $\frac{1}{q^{2}-1}$ and $\frac{q}{q^{2}-1}$ is that $T_{q,0}(\frac{1}{q^{2}-1})=\frac{q}{q^{2}-1}$ and $T_{q,1}(\frac{q}{q^{2}-1})=\frac{1}{q^{2}-1}.$ If $y \in (\frac{1}{q^{2}-1},\frac{q}{q^{2}-1}),$ then the following identities hold:
\begin{equation}
\label{equation 1}
T_{q,1}(T_{q,0}(y))-\frac{1}{q^{2}-1}=q^{2}\Big(y-\frac{1}{q^{2}-1}\Big) \textrm{ and } \frac{q}{q^{2}-1}-T_{q,0}(T_{q,1}(y))= q^{2}\Big(\frac{q}{q^{2}-1}-y\Big),
\end{equation} i.e., $T_{q,1}\circ T_{q,0}$ scales the distance between $y$ and $\frac{1}{q^{2}-1}$ by a factor $q^{2}$ and $T_{q,0}\circ T_{q,1}$ scales the distance between $y$ and $\frac{q}{q^{2}-1}$ by a factor $q^{2}.$

Returning to $a(x)\in S_{q}$, if $a(x)\in J_{q}$ then we are done, let us suppose this is not the case and $a(x)\in S_{q}\setminus J_{q}= [\frac{1}{q}, \frac{q+q^{2}}{q^{4}-1})\cup (\frac{1+q^{3}}{q^{4}-1},\frac{1}{q(q-1)}].$ If $a(x)$ in $[\frac{1}{q}, \frac{q+q^{2}}{q^{4}-1})$ then it follows from (\ref{period 4 identities}), (\ref{equation 1}) and the monotonicity of the maps $T_{q,0}$ and $T_{q,1}$ that sufficiently many iterates of the map $T_{q,1}\circ T_{q,0}$ will map $a(x)$ into $J_{q}$, similarly if $a(x)\in (\frac{1+q^{3}}{q^{4}-1},\frac{1}{q(q-1)}]$ then sufficiently many iterates of the map $T_{q,0}\circ T_{q,1}$ will map $a(x)$ into $J_{q}.$
\end{proof}

To prove $(\frac{1+\sqrt{5}}{2},q_{\aleph_{0}})\cap \mathcal{B}_{\aleph_{0}}=\emptyset$ it is necessary to determine which elements of $J_{q}$ are preimages of points with unique $q$-expansion, these points are classified in the following proposition.

\begin{prop}
\label{Branching points prop}
Let $q\in(\frac{1+\sqrt{5}}{2},q_{\aleph_{0}}),$ then $$(T^{-1}_{q,0}(U_{q})\cap J_{q})\cup (T^{-1}_{q,1}(U_{q})\cap J_{q})=\{(01(10)^{\infty})_{q}, (011(10)^{\infty})_{q}, (10(01)^{\infty})_{q}, (100(01)^{\infty})_{q}\}.$$
\end{prop}
\begin{proof}
By Lemma \ref{Unique expansions lemma} to prove our result it suffices to show that the following identities hold for $q\in(\frac{1+\sqrt{5}}{2},q_{\aleph_{0}})$:
\begin{enumerate}
	\item $T_{q,0}(\frac{q+q^{2}}{q^{4}-1})\in (((10)^{\infty})_{q},(1(10)^{\infty})_{q})$
	\item $T_{q,0}(\frac{1+q^{3}}{q^{4}-1})\in ((11(10)^{\infty})_{q},(111(10)^{\infty})_{q})$
	\item $T_{q,1}(\frac{q+q^{2}}{q^{4}-1})\in ((000(01)^{\infty})_{q},(00(01)^{\infty})_{q})$
	\item $T_{q,1}(\frac{1+q^{3}}{q^{4}-1})\in ((0(01)^{\infty},(01)^{\infty})_{q}).$
\end{enumerate}
Performing several straightforward calculations we can show that for $q\in(\frac{1+\sqrt{5}}{2},q_{\aleph_{0}})$ each of these identities hold. We remark that the upper bound $q_{\aleph_{0}}$ is optimal as $$T_{q_{\aleph_{0}},0}\Big(\frac{1+q_{\aleph_{0}}^{3}}{q_{\aleph_{0}}^{4}-1}\Big)=(111(10)^{\infty})_{q_{\aleph_{0}}} \textrm{ and }T_{q_{\aleph_{0}},1}\Big(\frac{q_{\aleph_{0}}+q_{\aleph_{0}}^{2}}{q_{\aleph_{0}}^{4}-1}\Big)=(000(01)^{\infty})_{q_{\aleph_{0}}}.$$

\end{proof}

We are now in a position to prove $(\frac{1+\sqrt{5}}{2},q_{\aleph_{0}})\cap \mathcal{B}_{\aleph_{0}}=\emptyset.$ Suppose $q\in (\frac{1+\sqrt{5}}{2},q_{\aleph_{0}})\cap \mathcal{B}_{\aleph_{0}},$ then by Proposition \ref{Branching prop} we may assume $x\in (0,\frac{1}{q-1})$ is a $q$ null infinite point. We let $a^{1}=(a^{1}_{1},\ldots ,a^{1}_{k_{1}})$ denote the unique minimal branching sequence of $x,$ since $\text{card }\Omega_{q}(T_{q,i}(a^{1}(x)))=1$ for some $i\in\{0,1\},$ there exists a unique minimal sequence of transformations $a^{2}=(a^{2}_{1},\ldots, a^{2}_{k_{2}})$ such that $a^{2}(a^{1}(x))\in S_{q}$. Similarly, for $i\geq 2$ we let $a^{i}=(a^{i}_{1},\ldots ,a^{i}_{k_{i}})$ denote the unique minimal sequence of transformations satisfying $a^{i}(a^{i-1}(\ldots(a^{1}(x))\ldots))\in S_{q}.$ We refer the reader to Figure \ref{fig3} for a diagram depicting the branching tree corresponding to $x$ when $x$ is a $q$ null infinite point and $q\in(\frac{1+\sqrt{5}}{2},q_{\aleph_{0}}),$ this diagram illustrates the role the of the sequence $(a^{i})_{i=1}^{\infty}$. For ease of exposition we denote the finite concatenation $a^{1}a^{2}\cdots a^{i}$ by $b^{i},$ therefore $a^{i}(a^{i-1}(\ldots(a^{1}(x))\ldots))=b^{i}(x).$ For $q\in (\frac{1+\sqrt{5}}{2},2),$ if $b^{i}(x)\in S_{q}$ then $T_{q,i}(b^{i}(x))\notin S_{q}$ for $i\in\{0,1\},$ this implies $k_{i}\geq 2$ for all $i\in \mathbb{N}.$ 
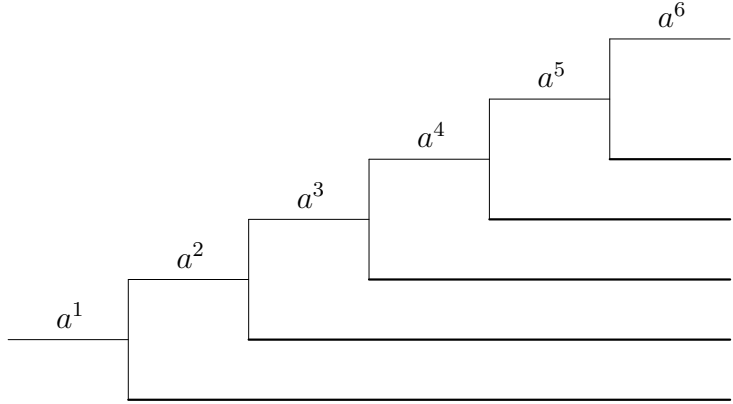
\begin{figure}[t]
\centering \unitlength=0.8mm
\begin{picture}(100,80)(0,0)
\thinlines
\path(-20,20)(0,20)
\path(0,30)(0,10)
\path(0,30)(20,30)
\path(20,40)(20,20)
\path(20,40)(40,40)
\path(40,50)(40,30)
\path(40,50)(60,50)
\path(60,60)(60,40)
\path(60,60)(80,60)
\path(80,70)(80,50)
\path(80,70)(100,70)
\put(-12,22){$a^{1}$}
\put(8,32){$a^{2}$}
\put(28,42){$a^{3}$}
\put(48,52){$a^{4}$}
\put(68,62){$a^{5}$}
\put(88,72){$a^{6}$}
\thicklines
\path(0,10)(100,10)
\path(20,20)(100,20)
\path(40,30)(100,30)
\path(60,40)(100,40)
\path(80,50)(100,50)
\end{picture}
\caption{The branching tree corresponding to $x$ when $x$ is a $q$ null infinite point and $q\in(\frac{1+\sqrt{5}}{2},q_{\aleph_{0}})$}
    \label{fig3}
\end{figure}

By Lemma \ref{inner switch lemma} we can assert that $b^{n}(x)\in J_{q}$ for some $n\geq 2$. Since $x$ is a $q$ null infinite point Proposition \ref{Branching points prop} implies $b^{n}(x)\in\{(01(10)^{\infty})_{q}, (011(10)^{\infty})_{q}, (10(01)^{\infty})_{q}, (100(01)^{\infty})_{q}\}.$ We now show that if $b^{n}(x)\in \{(01(10)^{\infty})_{q}, (011(10)^{\infty})_{q}, (10(01)^{\infty})_{q}, (100(01)^{\infty})_{q}\},$ then $T_{q,i}(b^{n-1}(x))\notin U_{q},$ for $i\in\{0,1\}.$ This will contradict our assumption that $x$ is a $q$ null infinite point and implies $(\frac{1+\sqrt{5}}{2},q_{\aleph_{0}})\cap \mathcal{B}_{\aleph_{0}}=\emptyset$.

If $a^{n}=(a^{n}_{1},\ldots ,a^{n}_{k_{n}})$ then without loss in generality we may assume that $a^{n}_{k_{n}}=T_{q,0}.$ The following lemma determines $a^{n}_{i}$ for $1\leq i< k_{n}$.
\begin{lemma}
\label{sequence lemma}
Let $x$ and $a^{n}$ be as above. Suppose $a^{n}_{k_{n}}=T_{q,0},$ then $a^{n}_{1}=T_{q,1}$ and $a^{n}_{i}= T_{q,0}$ for $1<i<k_{n}.$
\end{lemma}
\begin{proof}
We begin by showing $a^{n}_{1}=T_{q,1}$. We suppose $a^{n}_{1}=T_{q,0}$ and derive a contradiction. If $a^{n}_{1}=T_{q,0},$ then $T_{q,0}(b^{n-1}(x))\in (\frac{1}{q(q-1)},\frac{1}{q-1}]$ and $k_{n}\geq 3.$ Letting $a'=(a^{n}_{2},\ldots, a^{n}_{k_{n}-1})$ it follows that $a'(T_{q,0}(b^{n-1}(x)))= T^{-1}_{q,0}(a^{n}(x))$ and by the minimality of $a^{n}$ we must have $$(a^{n}_{i}\circ \cdots \circ a^{n}_{2})(T_{q,0}(b^{n-1}(x)))\notin S_{q}$$ for $2\leq i \leq k_{n}-1$, we now show that this is not possible. 

If $a'(T_{q,0}(b^{n-1}(x)))= T^{-1}_{q,0}(a^{n}(x))$ and $(a^{n}_{i}\circ \cdots \circ a^{n}_{2})(T_{q,0}(b^{n-1}(x)))\notin S_{q}$ for all $2\leq i \leq k_{n}-1$, then it is a consequence of $T_{q,1}$ being strictly decreasing on $(0,\frac{1}{q-1})$ and $T_{q,0}$ being strictly increasing on $(0,\frac{1}{q-1})$ that we must have
\begin{equation}
\label{Case equation}
T_{q,1}\Big(\frac{1}{q(q-1)}\Big)\leq T^{-1}_{q,0}(b^{n}(x)).
\end{equation}
By our assumption $b^{n}(x)\in\{(01(10)^{\infty})_{q}, (011(10)^{\infty})_{q}, (10(01)^{\infty})_{q}, (100(01)^{\infty})_{q}\},$ therefore to derive our contradiction it suffices to show that (\ref{Case equation}) does not hold for each of these four cases. As such, to conclude $a^{n}_{1}=T_{q,1}$ we need to show that following inequalities hold:

\begin{enumerate}
	\item $T^{-1}_{q,0}((01(10)^{\infty})_{q})<T_{q,1}\Big(\frac{1}{q(q-1)}\Big)$
	\item $T^{-1}_{q,0}((011(10)^{\infty})_{q})<T_{q,1}\Big(\frac{1}{q(q-1)}\Big)$
	\item $T^{-1}_{q,0}((10(01)^{\infty})_{q})<T_{q,1}\Big(\frac{1}{q(q-1)}\Big)$
	\item $T^{-1}_{q,0}((100(01)^{\infty})_{q})<T_{q,1}\Big(\frac{1}{q(q-1)}\Big).$
\end{enumerate} By the monotonicity of the map $T^{-1}_{q,0}$ it suffices to show only (2) and (3) hold. We can show that (2) holds for $q\in(\frac{1+\sqrt{5}}{2},1.69765\ldots),$ here $1.69765\ldots$ is the appropriate root of $x^6=x^5+2x^4-x^3-x^2+1.$ Similarly (3) holds for $q\in(\frac{1+\sqrt{5}}{2},1.68042\ldots)$ where $1.68042\ldots$ is the appropriate root of $x^5=x^4+x^3+x-1.$ Therefore (\ref{Case equation}) does not hold and we may conclude $a^{n}_{1}=T_{q,1}.$ 

It remains to show $a^{n}_{i}= T_{q,0}$ for $1<i<k_{n}.$ Suppose $a^{n}_{i}= T_{q,1}$ for some $1<i<k_{n},$ then either $$(a^{n}_{i-1}\circ \cdots \circ a^{n}_{1})(b^{n-1}(x))\in S_{q}\textrm{ or }(a^{n}_{i-1}\circ \cdots \circ a^{n}_{1})(b^{n-1}(x))\in \Big(\frac{1}{q(q-1)},\frac{1}{q-1}\Big].$$ As a consequence of the minimality of $a^{n}$ we cannot have $(a^{n}_{i-1}\circ \cdots \circ a^{n}_{1})(b^{n-1}(x))\in S_{q}.$ By analogous reasoning to that stated in the first part of our proof, the minimality of $a^{n}$ also implies that we cannot have $(a^{n}_{i-1}\circ \cdots \circ a^{n}_{1})(b^{n-1}(x))\in (\frac{1}{q(q-1)},\frac{1}{q-1}].$ We may therefore conclude $a^{n}_{i}= T_{q,0}$ for all $1<i<k_{n}.$
\end{proof}
By Lemma \ref{sequence lemma} we have $b^{n}(x)= (T^{k_{n}}_{q,0}\circ T_{q,1})(b^{n-1}(x)),$ since $x$ is a $q$ null infinite point we have $T_{q,0}(b^{n-1}(x))\in U_{q}.$ This is equivalent to 
\begin{equation}
\label{uniqueness equation}
T^{-k_{n}}_{q,0}(b^{n}(x))+1\in U_{q}.
\end{equation} To derive our contradiction we will show that (\ref{uniqueness equation}) cannot occur for all $q\in (\frac{1+\sqrt{5}}{2},q_{\aleph_{0}})$. Since $b^{n}(x)\in\{(01(10)^{\infty})_{q}, (011(10)^{\infty})_{q}, (10(01)^{\infty})_{q}, (100(01)^{\infty})_{q}\}$ there are four cases to consider, the analysis of these cases is summarised in the following proposition.

\begin{prop}
\label{First half prop}
For $q\in(\frac{1+\sqrt{5}}{2},q_{\aleph_{0}})$ the following inequalities hold:
\begin{enumerate}
	\item $T^{-1}_{q,0}((01(10)^{\infty})_{q})+1\in ((111(10)^{\infty})_{q},(1111(10)^{\infty})_{q})$
	\item $T^{-2}_{q,0}((01(10)^{\infty})_{q})+1\in ((1(10)^{\infty})_{q},(11(10)^{\infty})_{q})$
	\item $T^{-j}_{q,0}((01(10)^{\infty})_{q})+1\in (((10)^{\infty}),(1(10)^{\infty})_{q})$ for all $j\geq 3$
	\item $T^{-1}_{q,0}((011(10)^{\infty})_{q})+1\in ((1111(10)^{\infty})_{q},(11111(10)^{\infty})_{q})$
	\item $T^{-2}_{q,0}((011(10)^{\infty})_{q})+1\in ((1(10)^{\infty})_{q},(11(10)^{\infty})_{q})$
	\item $T^{-j}_{q,0}((011(10)^{\infty})_{q})+1\in (((10)^{\infty})_{q},(1(10)^{\infty})_{q})$ for all $j\geq 3$
	\item $T^{-1}_{q,0}((100(01)^{\infty})_{q})+1\in ((111(10)^{\infty})_{q},(1111(10)^{\infty})_{q})$
	\item $T^{-2}_{q,0}((100(01)^{\infty})_{q})+1\in ((1(10)^{\infty})_{q},(11(10)^{\infty})_{q})$
	\item $T^{-j}_{q,0}((100(01)^{\infty})_{q})+1\in (((10)^{\infty}),(1(10)^{\infty})_{q})$ for all $j\geq 3$
	\item $T^{-1}_{q,0}((10(01)^{\infty})_{q})+1\in ((1111(10)^{\infty})_{q},(11111(10)^{\infty})_{q})$
	\item $T^{-2}_{q,0}((10(01)^{\infty})_{q})+1\in ((1(10)^{\infty})_{q},(11(10)^{\infty})_{q})$
	\item $T^{-j}_{q,0}((10(01)^{\infty})_{q})+1\in (((10)^{\infty})_{q},(1(10)^{\infty})_{q})$ for all $j\geq 3$
\end{enumerate}
\end{prop}
\begin{proof}
Showing that these identities hold is a simple yet time consuming exercise, as such we omit the details. Our calculations yielded the following:
\begin{itemize}
	\item (1), (2) and (3) hold for $q\in(\frac{1+\sqrt{5}}{2},1.67602\ldots)$, where $1.67602\ldots$ is the appropriate root of $x^5=2x^3+x^2+1$
	\item (4), (5) and (6) hold for $q\in(\frac{1+\sqrt{5}}{2},1.65462\ldots)$, where $1.65462\ldots$ is the appropriate root of $x^6=2x^4+x^3+1$
	\item (7), (8) and (9) hold for $q\in(\frac{1+\sqrt{5}}{2},1.666184\ldots)$, where $1.66184\ldots$ is the appropriate root of $x^5=x^3+x^2+2x+2$
	\item (10), (11) and (12) hold for $q\in(\frac{1+\sqrt{5}}{2},q_{\aleph_{0}}).$
\end{itemize}
\end{proof}
By Proposition \ref{First half prop} we can deduce that (\ref{uniqueness equation}) does not hold and we have our desired contradiction, we may therefore conclude $(\frac{1+\sqrt{5}}{2},q_{\aleph_{0}})\cap \mathcal{B}_{\aleph_{0}}=\emptyset.$

\subsection{Proof that $q_{\aleph_{0}}\in \mathcal{B}_{\aleph_{0}}$}
By the above remarks to conclude Theorem \ref{Main thm} it suffices to show that $q_{\aleph_{0}}\in \mathcal{B}_{\aleph_{0}}.$ The proof of this statement is contained within the following proposition.

\begin{prop}
\label{second half prop}
$\frac{q_{\aleph_{0}}+q_{\aleph_{0}}^{2}}{q_{\aleph_{0}}^{4}-1}$ and $\frac{1+q_{\aleph_{0}}^{3}}{q_{\aleph_{0}}^{4}-1}$ have countably infinite $q_{\aleph_{0}}$-expansions.
\end{prop}
\begin{proof}
To begin with we recall that $$T_{q,1}\Big(T_{q,0}\Big(\frac{q+q^{2}}{q^{4}-1}\Big)\Big)= \frac{1+q^{3}}{q^{4}-1}\textrm{ and } T_{q,0}\Big(T_{q,1}\Big(\frac{1+q^{3}}{q^{4}-1}\Big)\Big)= \frac{q+q^{2}}{q^{4}-1},$$ for all $q\in (1,2)$. As stated in the proof of Proposition \ref{Branching points prop} $$T_{q_{\aleph_{0}},0}\Big(\frac{1+q_{\aleph_{0}}^{3}}{q_{\aleph_{0}}^{4}-1}\Big)=(111(10)^{\infty})_{q_{\aleph_{0}}} \textrm{ and }T_{q_{\aleph_{0}},1}\Big(\frac{q_{\aleph_{0}}+q_{\aleph_{0}}^{2}}{q_{\aleph_{0}}^{4}-1}\Big)=(000(01)^{\infty})_{q_{\aleph_{0}}}.$$ Since $$T_{q_{\aleph_{0}},0}\Big(\frac{q_{\aleph_{0}}+q_{\aleph_{0}}^{2}}{q_{\aleph_{0}}^{4}-1}\Big)\in \Big(\frac{1}{q_{\aleph_{0}}(q_{\aleph_{0}}-1)},\frac{1}{q_{\aleph_{0}}-1}\Big]\textrm{ and } T_{q_{\aleph_{0}},1}\Big(\frac{1+q_{\aleph_{0}}^{3}}{q_{\aleph_{0}}^{4}-1}\Big)\in\Big[0,\frac{1}{q_{\aleph_{0}}}\Big),$$ it follows that $$\Sigma_{q_{\aleph_{0}}}\Big(\frac{1+q_{\aleph_{0}}^{3}}{q_{\aleph_{0}}^{4}-1}\Big)=\Big\{(1001)^{\infty}, (1001)^{k}0111(10)^{\infty}, (1001)^{k}101000(01)^{\infty}| \textrm{ for some } k\geq 0\Big\}$$ and $$\Sigma_{q_{\aleph_{0}}}\Big(\frac{q_{\aleph_{0}}+q_{\aleph_{0}}^{2}}{q_{\aleph_{0}}^{4}-1}\Big)=\Big\{(0110)^{\infty},(0110)^{k}1000(01)^{\infty}, (0110)^{k}010111(10)^{\infty} | \textrm{ for some } k\geq 0\Big\}.$$
\end{proof} It is immediate from the proof of Proposition \ref{second half prop} that both $\frac{q_{\aleph_{0}}+q_{\aleph_{0}}^{2}}{q_{\aleph_{0}}^{4}-1}$ and $\frac{1+q_{\aleph_{0}}^{3}}{q_{\aleph_{0}}^{4}-1}$ are $q_{\aleph_{0}}$ null infinite points. By Proposition \ref{second half prop} we have $q_{\aleph_{0}}\in \mathcal{B}_{\aleph_{0}}$ and by our earlier remarks we may conclude Theorem \ref{Main thm}. 
\section{General results}
\label{Final discussion}
In this section we shall prove some general results that arose from our proof of Theorem \ref{Main thm}.
\subsection{The continuum hypothesis for $\Sigma_{q}(x)$}
In this section we show that the following theorem holds.
\begin{thm}
\label{continuum theorem}
Let $q\in(1,2)$ and $x\in I_{q},$ if $\Sigma_{q}(x)$ is uncountable then $\text{card }\Sigma_{q}(x)=2^{\mathbb{\aleph}_{0}}.$
\end{thm} To prove this statement we need to construct another infinite tree in a similar way to how we constructed the branching tree corresponding to $x$ and the infinite branching tree corresponding to $x.$ We define the \textit{$2^{\aleph_{0}}$ branching tree corresponding to $x$} as follows. Suppose $x\in I_{q}$ satisfies $\Omega_{q}(x)$ is uncountable, if for each branching point of $x$ we have $\text{card }\Omega_{q}(T_{q,i}(a(x)))\leq \aleph_{0}$ for some $i\in\{0,1\},$ then the $2^{\aleph_{0}}$ branching tree corresponding to $x$ is an infinite horizontal line. If this is not the case then there exists a unique minimal branching sequence $a$ such that $\Omega_{q}(T_{q,0}(a(x)))$ and $\Omega_{q}(T_{q,1}(a(x)))$ are both uncountable, in this case we draw a finite horizontal line that then bifurcates with upper branch corresponding to $T_{q,0}(a(x))$ and lower branch corresponding to $T_{q,1}(a(x)).$ Applying these rules to the branches corresponding to $T_{q,0}(a(x)),$ $T_{q,1}(a(x))$ and all subsequent branches we obtain an infinite tree.  We refer to the infinite tree we obtain through this construction as the $2^{\aleph_{0}}$ branching tree corresponding to $x.$ Where appropriate we denote the $2^{\aleph_{0}}$ branching tree corresponding to $x$ by $\mathcal{T}_{2^{\aleph_{0}}}(x).$ 

\begin{remark}
\label{remark 1}
As was the case for $\mathcal{T}(x)$ and $\mathcal{T}_{\infty}(x)$ each infinite path in $\mathcal{T}_{2^{\aleph_{0}}}(x)$ can be identified with a unique element of $\Omega_{q}(x).$
\end{remark}
By Remark \ref{remark 1} to prove that if $\Sigma_{q}(x)$ is uncountable then $\text{card }\Sigma_{q}(x)=2^{\mathbb{\aleph}_{0}}$ it suffices to show that $\mathcal{T}_{2^{\aleph_{0}}}(x)$ is always the full binary tree. We will show that whenever $x\in I_{q}$ satisfies $\Sigma_{q}(x)$ is uncountable then there exists a branching sequence for $x$ such that $\Omega_{q}(T_{q,0}(a(x)))$ and $\Omega_{q}(T_{q,1}(a(x)))$ are both uncountable. Repeatedly applying this result to succesive branches in our construction will imply that every branch bifurcates and that $\mathcal{T}_{2^{\aleph_{0}}}(x)$ is the full binary tree.
\begin{lemma}
\label{uncountable lemma}
Let $q\in(1,2)$ and $x\in I_{q}.$ If $\Omega_{q}(x)$ is uncountable or equivalently $\Sigma_{q}(x)$ is uncountable, then there exists a branching point of $x$, $a(x),$ such that $\Omega_{q}(T_{q,0}(a(x)))$ and $\Omega_{q}(T_{q,1}(a(x)))$ are both uncountable.
\end{lemma}
\begin{proof}
Suppose that for every branching point of $x$ we have $\textrm{card }\Omega_{q}(T_{q,i}(a(x)))\leq \aleph_{0}$ for some $i\in\{0,1\}$. We let $a^{1}=(a^{1}_{1},\ldots, a^{1}_{n_{1}})$ denote the unique minimal branching sequence of $x,$ by our assumption $\textrm{card }\Omega_{q}(T_{q,i_{1}}(a^{1}(x)))\leq \aleph_{0}$ for some $i_{1}\in\{0,1\},$ as $\Omega_{q}(x)$ is uncountable we must have $\Omega_{q}(T_{q,1-i_{1}}(a^{1}(x)))$ is uncountable. It is a consequence of $\Omega_{q}(T_{q,1-i_{1}}(a^{1}(x)))$ being  uncountable and our assumption, that there exists a unique minimal branching sequence $a^{2}=(a^{2}_{1},\ldots, a^{2}_{n_{2}})$ and $i_{2}\in\{0,1\}$ satisfying: $n_{2}>n_{1},$ $a^{1}_{j}=a^{2}_{j}$ for $1\leq j\leq n_{1},$ $\Omega_{q}(a^{2}(x))$ is uncountable, $\text{card }\Omega_{q}(T_{q,i_{2}}(a^{2}(x)))\leq \aleph_{0}$ and $\Omega_{q}(T_{q,1-i_{2}}(a^{2}(x)))$ is uncountable. Moreover, for $k\geq 2$ we define $a^{k}=(a^{k}_{1},\ldots, a^{k}_{n_{k}})$ and $i_{k}\in\{0,1\}$ inductively as follows, let $a^{k}$ denote the the unique minimal branching sequence of $x$ such that $n_{k}>n_{k-1},$ $a^{k-1}_{j}=a^{k}_{j}$ for $1\leq j\leq n_{k-1}$ and $\Omega_{q}(a^{k}(x))$ is uncountable, we let $i_{k}$ denote the unique element of $\{0,1\}$ such that $\text{card }\Omega_{q}(T_{q,i_{k}}(a^{k}(x)))\leq \aleph_{0}$ and $\Omega_{q}(T_{q,1-i_{k}}(a^{k}(x)))$ is uncountable.

To each $a^{k}$ we associate the set $$\Omega_{a^{k}}(x)=\Big\{ a\in \Omega_{q}(x)| a_{j}=a^{k}_{j} \text{ for } 1\leq j \leq n_{k} \text{ and } a_{n_{k}+1}= T_{q,i_{k}}\Big\}.$$Clearly $\text{card }\Omega_{a^{k}}(x)\leq \aleph_{0}$. Letting $a^{\infty}\in\{T_{q,0},T_{q,1}\}^{\mathbb{N}}$ denote the unique infinite sequence obtained as the componentwise limit of $(a^{k})_{k=1}^{\infty},$ it is an immediate consequence of our construction that $$\Omega_{q}(x)=\{a^{\infty}\}\cup(\bigcup_{k=1}^{\infty} \Omega_{a^{k}(x)})$$ and that $\text{card }\Omega_{q}(x)\leq \aleph_{0}$, a contradiction. Therefore there must exists a branching point of $x$ such that both $\Omega_{q}(T_{q,0}(a(x)))$ and $\Omega_{q}(T_{q,1}(a(x)))$ are uncountable.

\end{proof}
Theorem \ref{continuum theorem} follows from our earlier remarks.

\subsection{Properties of $\mathcal{B}_{\aleph_{0}}\cap ([\frac{1+\sqrt{5}}{2},q_{f})\setminus\{q_{2}\})$}
It is clear from the proof of Theorem \ref{Main thm} that the interval $J_{q}$ is an appropriate object of study, in particular, we are interested in its subset $(T^{-1}_{q,0}(U_{q})\cap J_{q})\cup (T^{-1}_{q,1}(U_{q})\cap J_{q}).$ For $k\geq 3$ we let $\alpha_{k}$ denote the unique $q\in(1,2)$ such that $$T_{q,0}\Big(\frac{1+q^{3}}{q^{4}-1}\Big)=((1)^{k}(10)^{\infty})_{q},$$ the appropriate root of $x^{k+4}=x^{k+3}+x^{k+2}+x^{k}-x^{2}-1.$ In particular $\alpha_{3}=q_{\aleph_{0}}$. It is a simple exercise to show that $\alpha_{k}\in[q_{\aleph_{0}},q_{f})$ for all $k\geq 3$ and $\alpha_{k}\nearrow q_{f}.$ Adapting the proof of Proposition \ref{second half prop} it can be shown that $\alpha_{k}\in \mathcal{B}_{\aleph_{0}},$ for all $k\geq 3.$ The significance of $\alpha_{k}$ follows from the fact that for $q\in[\alpha_{k},\alpha_{k+1})$ we have 
\begin{equation}
\label{alpha equation}
(T^{-1}_{q,0}(U_{q})\cap J_{q})\cup (T^{-1}_{q,1}(U_{q})\cap J_{q})=\Big\{(1(0)^{j}(01)^{\infty})_{q}, (0(1)^{j}(10)^{\infty})_{q}| \textrm{ for } 1\leq j\leq k\Big\}.
\end{equation} In what follows we let $$P_{q}=(T^{-1}_{q,0}(U_{q})\cap J_{q})\cup (T^{-1}_{q,1}(U_{q})\cap J_{q})$$ and $$U_{k,q}=\Big\{(1(0)^{j}(01)^{\infty})_{q}, (0(1)^{j}(10)^{\infty})_{q}| \textrm{ for } 1\leq j\leq k\Big\}.$$ The following result is implicit in our proof of Theorem \ref{Main thm} and therefore stated without proof.

\begin{prop}
\label{iff prop}
Let $q\in [q_{\aleph_{0}},q_{f})\setminus\{q_{2}\},$ then $q\in \mathcal{B}_{\aleph_{0}}$ if and only if $P_{q}$ contains a $q$ null infinite point.
\end{prop}
Suppose $q\in [q_{\aleph_{0}},q_{f})\setminus\{q_{2}\},$ then $q\in [\alpha_{k},\alpha_{k+1})$ for some $k\geq 3,$ it follows from (\ref{alpha equation}) and Proposition \ref{iff prop} that to determine whether $q\in \mathcal{B}_{\aleph_{0}}$ we only have to verify whether $U_{k,q}$ contains a $q$ null infinite point. This statement makes determining whether $q\in \mathcal{B}_{\aleph_{0}}$ a reasonably straightforward computation as we only have finitely many cases to consider. Proposition \ref{iff prop} also yields the following result.
\begin{thm}
\label{discrete thm}
$\mathcal{B}_{\aleph_{0}}\cap ([\frac{1+\sqrt{5}}{2},q_{f})\setminus\{q_{2}\})$ is a discrete set.
\end{thm}
\begin{proof}
As $\mathcal{B}_{\aleph_{0}}\cap[\frac{1+\sqrt{5}}{2},q_{\aleph_{0}})=\{\frac{1+\sqrt{5}}{2}\}$ it suffices to show that $\mathcal{B}_{\aleph_{0}}\cap ([q_{\aleph_{0}},q_{f})\setminus\{q_{2}\})$ is a discrete set. For each $q^{*}\in \mathcal{B}_{\aleph_{0}}\cap ([q_{\aleph_{0}},q_{f})\setminus\{q_{2}\}),$ we shall construct an open interval $I_{q^{*}}$ satisfying: $q^{*}\in I_{q^{*}}$ and $(I_{q^{*}}\setminus\{q^{*}\})\cap \mathcal{B}_{\aleph_{0}}=\emptyset,$  this will imply $\mathcal{B}_{\aleph_{0}}\cap ([q_{\aleph_{0}},q_{f})\setminus\{q_{2}\})$ is a  discrete set.
  
Suppose $q^{*}\in \mathcal{B}_{\aleph_{0}}\cap ([q_{\aleph_{0}},q_{f})\setminus\{q_{2}\}),$ then $q^{*}\in[\alpha_{k},\alpha_{k+1})$ for some $k\geq 3$ and $P_{q^{*}}=U_{k,q^{*}}.$ By a continuity argument there exists an open interval $I_{1}$ satisfying: $q^{*}\in I_{1}$ and $P_{q}\subseteq U_{k,q},$ for all $q\in I_{1}.$ We let $$\Sigma_{null}=\Big\{ (\epsilon_{i})_{i=1}^{\infty}\in \{1(0)^{j}(01)^{\infty}, 0(1)^{j}(10)^{\infty}| 1\leq j\leq k\}\Big| ((\epsilon_{i})_{i=1}^{\infty})_{q^{*}} \textrm{ is a $q^{*}$ null infinite point }\Big\},$$ and $$\Sigma_{bif}=\Big\{1(0)^{j}(01)^{\infty}, 0(1)^{j}(10)^{\infty}| 1\leq j\leq k\Big\}\setminus \Sigma_{null}.$$ For ease of exposition we let $\Sigma_{null}=\{(\epsilon_{i}^{m})_{i=1}^{\infty}\}_{m=1}^{M}$ and $\Sigma_{bif}=\{(\epsilon_{i}^{n})_{i=1}^{\infty}\}_{n=1}^{N}.$ We will show that for each $(\epsilon_{i}^{m})_{i=1}^{\infty}\in \Sigma_{null}$ there exists a finite sequence of transformations $a$ and an open interval $I_{m}$ such that, $q^{*}\in I_{m}$ and for each $q\in I_{m}\setminus\{q^{*}\}$ we have $T_{q,i}(a(((\epsilon_{i}^{m})_{i=1}^{\infty})_{q}))\notin U_{q},$ for $i\in\{0,1\}$. Similarly, we will show that for each $(\epsilon_{i}^{n})_{i=1}^{\infty}\in \Sigma_{bif}$ there exists a finite sequence of transformations $a$ and an open interval $I_{n}$ such that, $q^{*}\in I_{n}$ and for all $q\in I_{n}$ we have $T_{q,i}(a(((\epsilon_{i}^{n})_{i=1}^{\infty})_{q}))\notin U_{q},$ for $i\in\{0,1\}.$ Taking $$I_{q^{*}}= I_{1}\cap(\bigcap_{m=1}^{M} I_{m})\cap(\bigcap_{n=1}^{N}I_{n}),$$ it will follow from our construction that if $q\in I_{q^{*}}\setminus\{q^{*}\}$ then every element of $P_{q}$ cannot be a $q$ null infinite point, which by Proposition \ref{iff prop} implies $(I_{q^{*}}\setminus\{q^{*}\})\cap \mathcal{B}_{\aleph_{0}}=\emptyset$ and $\mathcal{B}_{\aleph_{0}}\cap ([q_{\aleph_{0}},q_{f})\setminus\{q_{2}\})$ is a discrete set.

To begin with let us consider $(\epsilon_{i}^{m})_{i=1}^{\infty}\in \Sigma_{null},$ by an application of Lemma \ref{inner switch lemma} there exists a finite sequence of transformations $a$ such that $a(((\epsilon_{i}^{m})_{i=1}^{\infty})_{q^{*}})\in P_{q^{*}},$  $T_{q^{*},i}(a(((\epsilon_{i}^{m})_{i=1}^{\infty})_{q^{*}}))\notin U_{q^{*}}$ and $T_{q^{*},1-i}(a(((\epsilon_{i}^{m})_{i=1}^{\infty})_{q^{*}}))=((\delta_{i})_{i=1}^{\infty})_{q^{*}}\in U_{q^{*}},$ for some $i\in\{0,1\}.$ By continuity we can assert that there exists an open interval $I'_{m}$ satisfying: $q^{*}\in I'_{m},$ $a(((\epsilon_{i}^{m})_{i=1}^{\infty})_{q}\in S_{q}$ and $T_{q,i}(a(((\epsilon_{i}^{m})_{i=1}^{\infty})_{q}))\notin U_{q}$ for all $q\in I'_{m}$. Since $q^{*}\in [\alpha_{k},\alpha_{k+1})$ we have $$(\delta_{i})_{i=1}^{\infty}\in \{(0)^{j}(01)^{\infty}, (1)^{j}(10)^{\infty}| 1\leq j\leq k\},$$ from which it follows that satisfying $T_{q,1-i}(a(((\epsilon_{i}^{m})_{i=1}^{\infty})_{q}))=((\delta_{i})_{i=1}^{\infty})_{q}$ is equivalent to satisfying $f(q)=0$ for some nontrivial polynomial $f(q)\in \mathbb{Z}[q].$ Clearly $f(q^{*})=0$, however, since $f(q)=0$ has a finite number of solutions there exists an open interval $I_{m}''$ satisfying: $q^{*}\in I_{m}''$, $a(((\epsilon_{i}^{m})_{i=1}^{\infty})_{q})\in S_{q}$ and $f(q)\neq 0$ for all $q\in I_{m}''\setminus\{q^{*}\}.$ Moreover, by continuity we may assume that $I''_{m}$ is sufficiently small such that $T_{q,1-i}(a(((\epsilon_{i}^{m})_{i=1}^{\infty})_{q}))\notin U_{q}\setminus\{((\delta_{i})_{i=1}^{\infty})_{q}\},$ for all $q\in I''_{m}$. Taking $I_{m}=I'_{m}\cap I''_{m},$ we may conclude that for all $q\in I_{m}\setminus\{q^{*}\}$ we have $T_{q,i}(a(((\epsilon_{i}^{m})_{i=1}^{\infty})_{q}))\notin U_{q},$ for $i\in\{0,1\}.$

It remains to consider $(\epsilon_{i}^{n})_{i=1}^{\infty}\in \Sigma_{bif},$ as $((\epsilon_{i}^{n})_{i=1}^{\infty})_{q^{*}}$ is not a $q^{*}$ null infinite point there exists a finite sequence of transformations $a$ such that $a(((\epsilon_{i}^{n})_{i=1}^{\infty})_{q^{*}})\in S_{q^{*}}$ and $T_{q^{*},i}(a(((\epsilon_{i}^{n})_{i=1}^{\infty})_{q^{*}}))\notin U_{q^{*}},$ for $i\in\{0,1\}.$ By continuity it follows that there exists an open interval $I_{n}$ such that, $q^{*}\in I_{n},$  $a(((\epsilon_{i}^{n})_{i=1}^{\infty})_{q})\in S_{q}$ and  $T_{q,i}(a(((\epsilon_{i}^{n})_{i=1}^{\infty})_{q}))\notin U_{q},$ for $i\in\{0,1\},$ for all $q\in I_{n}.$
\end{proof}
The discreteness of $\mathcal{B}_{\aleph_{0}}\cap ([\frac{1+\sqrt{5}}{2},q_{f})\setminus\{q_{2}\})$ leads to some interesting questions that we state in the next section.

\section{Open questions}
\label{Open problems}
To conclude we shall pose some open questions. 
\begin{itemize}
	\item In \cite{Sid1} Sidorov constructs a sequence $(q_{k})_{k=1}^{\infty}$ such that, $q_{k}\in \mathcal{B}_{\aleph_{0}}$ for all $k\geq 1$ and $q_{k}\searrow q_{2}.$ As stated at the start of Section \ref{Final discussion} $\alpha_{k}\nearrow q_{f},$ as such the following question seem natural. Suppose $q\in \mathcal{B}_{m}$ for some $m\geq 2,$ is $q$ a limit point of $\mathcal{B}_{\aleph_{0}}$? Moreover is the converse true, that is, if $q$ is a limit point of $\mathcal{B}_{\aleph_{0}}$ does that imply $q\in \mathcal{B}_{m}$ for some $m\geq 2?$ The discreteness of $\mathcal{B}_{\aleph_{0}}\cap([\frac{1+\sqrt{5}}{2},q_{f})\setminus\{q_{2}\})$ guaranteed by Theorem \ref{discrete thm} might seem to suggest so.
	\item Is $\mathcal{B}_{\aleph_{0}}$ closed?
	\item Is $q_{2}\in \mathcal{B}_{\aleph_{0}}$? If $q_{2}\in \mathcal{B}_{\aleph_{0}}$ then it would be a consequence of our above remarks, Theorem \ref{discrete thm} and \cite[Proposition~2.1]{BakerSid} that $\mathcal{B}_{\aleph_{0}}\cap [\frac{1+\sqrt{5}}{2},q_{f}]$ is a closed set.
	\item Given $q\in \mathcal{B}_{\aleph_{0}},$ what is the topology of the set of $q$ null infinite points?
\end{itemize}
\medskip\noindent {\bf Acknowledgments.} The author is grateful to Nikita Sidorov for his useful comments and insight.

\end{document}